\documentclass[12pt,reqno]{article}

\usepackage[top=2.5cm,bottom=2.5cm,left=2.5cm,right=2.5cm]{geometry}
\usepackage{graphicx}
\usepackage{latexsym}
\usepackage{amsmath,amssymb}
\usepackage{amscd}
\usepackage[arrow,matrix]{xy}
\usepackage{graphicx}
\usepackage{xcolor}
\usepackage{times}
\usepackage[]{subfigure}
 \usepackage{cite}
 \usepackage{float}
\usepackage{amsfonts}

\usepackage{amscd,amsmath,amsthm,amssymb}

\theoremstyle{plain}
\newtheorem{theorem}{Theorem}[section]

\theoremstyle{definition}

\newtheorem{definition}[theorem]{Definition}

\begin{document}

\title{\textbf{Curves of Constant Breadth According to Darboux Frame in a Strict Walker 3-Manifold}}
\author {{Ameth Ndiaye}\thanks{{ 
 E--mail: \texttt{ameth1.ndiaye@ucad.edu.sn} (A. Ndiaye)}}\\
 D\'epartement de Math\'ematiques, FASTEF, UCAD, Dakar, Senegal.}

\date{}
\maketitle%

\begin{abstract} 
In this paper, we investigate the differential geometry properties of curves of constant breadth according to Darboux frame in a given strict Walker 3-manifold. The considered curves are lying on a timelike surface in the Walker 3-manifold.
\end{abstract}
\begin{small} {\textbf{MSC:} 53B25 ; 53C40.}
\end{small}\\
\begin{small} {\textbf{Keywords:} Darboux frame, curvature, torsion, constant breadth curve, Walker 3-manifolds.} 
\end{small}\\
\maketitle

\section{Introduction}  
The study of curves of constant breadth were defined first in 1778 by Euler. Then, Solow \cite{Solow} investigated the curves of constant breadth. Kose, Magden and Yilmaz in \cite{Kose, Magden} studied plane curves of
constant breadth in Euclidean spaces $\mathbb{E}^3$ and $\mathbb{E}^4$.
Fujiwara \cite{Fuj7} defined constant breadth for space curves and obtained a problem
to determine whether there exists space curve of constant breadth or not. Furthermore, Blaschke \cite{Blas2} defined the curves of constant breadth on a sphere. In \cite{Altunkaya2}, Altunkaya et al.  defined null curves of constant breadth in Minkowski 4-space and obtain a characterization of these curves. Also Altunkaya et al. in \cite{Altunkaya1}  investigate constant breadth curves on a surface according to Darboux frame and give some characterizations of these curves.\\
Motivated by the above papers, we investigate the geometries of curves of constant breadth according to Darboux frame in a Strict Walker 3-manifold which is  a Lorentzian three-manifold admitting a parallel null vector field. It is known that Walker metrics have served as a powerful tool of constructing
interesting indefinite metrics which exhibit various aspects of geometric properties not given by any positive definite metrics. For more details about Walker 3-manifold see \cite{ Calvaruso, Diallo2, Gning}.

\section{Preliminaries}
A Walker $n$-manifold is a pseudo-Riemannian manifold, which admits 
a field of null parallel $r$-planes, with $r\leq\frac{n}{2}$. The 
canonical forms of the metrics were investigated by A. G. Walker 
(\cite{Brozos}). Walker  has derived adapted coordinates to a parallel 
plan field. Hence, the metric of a three-dimensional Walker manifold 
$(M, g_f^\epsilon)$ with coordinates $(x, y, z)$ is expressed as
\begin{eqnarray}\label{2.1}
g^\epsilon_f=dx\circ dz+\epsilon dy^2+f(x, y, z)dz^2
\end{eqnarray}
and its matrix form as 
\begin{eqnarray*}
g^{\epsilon}_{f}=\left(
\begin{array}{ccc}
0 & 0 & 1 \\
0 & \epsilon & 0  \\
1 & 0 & f
\end{array}\right)\,\,\,\, 
\text{with inverse}\,\,\, \, 
(g^{\epsilon}_{f})^{-1}=\left(
\begin{array}{ccc}
-f & 0 & 1 \\
0 & \epsilon & 0  \\
1 & 0 & 0 
\end{array}\right)
\end{eqnarray*}
for some function $f(x,y,z)$, where $\epsilon=\pm 1$ and thus 
$D = \mathrm{Span}{\partial_x}$ as the parallel degenerate line field.
Notice that when $\epsilon=1$ and $\epsilon=-1$ the Walker 
manifold has signature $(2,1)$ and $(1,2)$ respectively, and 
therefore is Lorentzian in both cases. In this study we take $\epsilon=1$. \\

It follows after a straightforward calculation that the Levi-Civita 
connection of any metric (\ref{2.1}) is given by:
\begin{eqnarray}\label{2.2}
\nabla_{\partial_x}\partial z &=& \frac{1}{2}f_x\partial_x, \quad 
\nabla_{\partial_y}\partial z=\frac{1}{2}f_y\partial_x,\nonumber\\
\nabla_{\partial_z}\partial z &=& \frac{1}{2}(ff_x+f_z)\partial_x
+ \frac{1}{2}f_y\partial_y-\frac{1}{2}f_x\partial_z
\end{eqnarray}
where $\partial_x$, $\partial_y$ and $\partial_z$ are the coordinate 
vector fields $\frac{\partial}{\partial_x}$, $\frac{\partial}{\partial_y}$ 
and $\frac{\partial}{\partial_z}$, respectively. Hence, if 
$(M, g_f^\epsilon)$ is a strict Walker manifolds i.e., $f(x,y,z)=f(y,z)$, 
then the associated Levi-Civita connection satisfies
\begin{eqnarray}
\nabla_{\partial_y}\partial z=\frac{1}{2}f_y\partial_x, \,\,\,\, \nabla_{\partial_z}\partial z=\frac{1}{2}f_z\partial_x-\frac{1}{2}f_y\partial_y.
\end{eqnarray}
Note that the existence of a null parallel vector field (i.e $f=f(y,z)$) 
simplifies the non-zero components of the Christoffel symbols and the 
curvature tensor of the metric $g^\epsilon_f$ as follows:
\begin{eqnarray}\label{2.5}
\Gamma_{23}^1 =\Gamma_{32}^1=\frac{1}{2}f_y,\;
\Gamma_{33}^1 = \frac{1}{2}f_z,\;
\Gamma_{33}^2 =-\frac{1}{2}f_y
\end{eqnarray} 

Let now $u$ and $v$ be two vectors in $M$. Denoted by $(\vec{i}, \vec{j}, \vec{k})$ 
the canonical frame in $\mathbb{R}^3$.\\
The vector product of $u$ and $v$ 
in $(M, g_f^\epsilon)$ with respect to the metric $g_f^\epsilon$ is the 
vector denoted by $u\times v$ in $M$ defined by
\begin{eqnarray}\label{2.7}
g_f^\epsilon( u\times v, w)=\det(u, v, w) 
\end{eqnarray}
for all vector $w$ in $M$, where $\det(u, v, w) $ is the determinant function associated to the canonical basis of $\mathbb{R}^3$. If $u=(u_1, u_2, u_3)$ 
and $v=(v_1, v_2, v_3)$ then by using (\ref{2.7}), we have:
\begin{eqnarray}\label{21}
u\times v= \left(
\left\vert\begin{matrix}
    u_1 & v_1\\
    u_2 & v_2
  \end{matrix}\right\vert-f
\left\vert\begin{matrix}
    u_2 & v_2 \\
    u_3 & v_3
  \end{matrix}\right\vert\right)\vec{i}
  -\epsilon\left\vert\begin{matrix}
    u_1 & v_1 \\
    u_3 & v_3
  \end{matrix}\right\vert \vec{j}+
  \left\vert\begin{matrix}
    u_2 & v_2 \\
    u_3 & v_3
  \end{matrix}\right\vert \vec{k}
\end{eqnarray}
\section{Darboux equations in Walker 3-manifold}
Let $\alpha: I\subset\mathbb{R}\longrightarrow (M, g_f^\epsilon)$ be a curve parametrized by its arc-length $s$.
The Frenet frame of $\alpha$ is the vectors $T$, $N$ and $B$ along $\alpha$ where $T$ is the tangent, $N$ the principal normal and $B$ the binormal vector. They satisfied the Frenet formulas
\begin{eqnarray}\label{FF}
{\left\{ \begin{array}{ccc}
  \nabla_TT(s) &=& \epsilon_2\kappa(s)N(s) \\
 \nabla_TN(s) &=& -\epsilon_1\kappa T(s)-\epsilon_3\tau B(s) \\
  \nabla_TB(s) &=& \epsilon_2\tau(s) N(s)
  \end{array}\right.}
\end{eqnarray}
where $\kappa$ and $\tau$ are respectively the curvature and the torsion of the curve $\alpha$, with $\epsilon_1=g_f(T;T); \epsilon_2=g_f(N;N)$ and $\epsilon_3=g_f(B,B)$.\\
Starting from local coordinates $(x, y, z)$ for which (\ref{2.1}) holds, it is easy to check that
\begin{eqnarray*}
e_1=\partial_y, \,\,\, e_2=\frac{2-f}{2\sqrt{2}}\partial_x+\frac{1}{\sqrt{2}}\partial_z, \,\,\, e_3=\frac{2+f}{2\sqrt{2}}\partial_x-\frac{1}{\sqrt{2}}\partial_z
\end{eqnarray*}
are local pseudo-orthonormal frame fields on $(M, g_f^\epsilon)$, with $g_f^\epsilon(e_1, e_1)=\epsilon$, $g_f^\epsilon(e_2, e_2)=1$ and $g_f^\epsilon(e_3, e_3)=-1$.
Thus the  signature of the metric $g_f^\epsilon$ is $(1, \epsilon, -1)$. If we choose $\epsilon=1$ then,  pseudo-orthonormal frame is formed by two spacelike vectors and one timelike vector and If we choose $\epsilon=-1$ then,  pseudo-orthonormal frame is formed by one spacelike vector and two timelike vectors. For both cases we obtain Lorentzian manifold. In this work we assume that $\epsilon=1$ \\
Now we suppose that the curve $\alpha$ lies on a timelike surface $S$ in $M$. Let $U$ be the unit normal vector of $S$, then the Darboux frame is given by $\lbrace T, Y, U\rbrace$, where $T$ is the tangent vector of the curve $\alpha(s)$ and $Y=U\times T$. \\
\textbf{Case 1:} Let $\alpha$ be timelike curve. Then the tangent vector $T$ is timelike ($\epsilon_1=-1$), the normal vector $N$ and the binormal vector $B$ are spacelike, that is ($\epsilon_2=\epsilon_3=1$).\\
Since $S$ is timelike, the unit normal vector $U$ is spacelike and so $Y$ becomes spacelike.
The usual transformations between the Walker Frenet frame and the Darboux takes the form
\begin{eqnarray}
Y=\cos\theta N+\sin\theta B\\
U=-\sin\theta N+\cos\theta B,
\end{eqnarray}
where $\theta$ is an angle between the vector $Y$ and the vector $N$.\\
Derivating $Y$ along the curve alpha we get 
\begin{eqnarray*}
\nabla_TY=\cos\theta\nabla_TN-\theta'\sin\theta N+\sin\theta\nabla_TB+\theta'\cos\theta B.
\end{eqnarray*}
Using the Frenet equation in (2.7) we have
\begin{eqnarray*}
\nabla_TY=\cos\theta(\kappa T-\epsilon_3\tau B)-\theta'\sin\theta N+\sin\theta(\epsilon_2\tau N)+\theta'\cos\theta B.
\end{eqnarray*}
Now we suppose that the principal normal and the binormal have the same sign. then we get
\begin{eqnarray}
\nabla_TY=\kappa\cos\theta T+(\theta'-\tau)U
\end{eqnarray}
The same calculus gives 
\begin{eqnarray}
\nabla_TU=-\kappa\sin\theta T-(\theta'-\tau)Y.
\end{eqnarray}
Then the Walker Darboux equation is expressed as
\begin{eqnarray}\label{DF1}
{\left\{ \begin{array}{ccc}
\nabla_TT=\kappa_g Y+\kappa_n U\\
\nabla_TY=\kappa_g T+\tau_gU\\
\nabla_TU=\kappa_n T-\tau_gY,
 \end{array}\right.}
\end{eqnarray}
where $\kappa_g$, $\kappa_n$ and $\tau_g$ are the geodesic curvature, normal curvature and geodesic torsion of $\alpha(s)$ on $S$, respectively. Also, (\ref{DF1}) gives
\begin{eqnarray}\label{eqf}
g^\epsilon_f\left(\nabla_TY,U\right)=\tau_g=\theta'-\tau, \,\,\,\,\,\,\,\,\\
g^\epsilon_f\left(\nabla_TT,Y\right)=\kappa_g=\kappa\cos\theta,  \,\,\,\,\,\,\,\,\\
g^\epsilon_f\left(\nabla_TT,U\right)=\kappa_n=-\kappa\sin\theta.
\end{eqnarray}
\textbf{Case 2:} Let $\alpha$ be spacelike curve. Then the tangent vector $T$ is spacelike ($\epsilon_1=1$), the normal vector $N$ is spacelike ($\epsilon_2=1$) and the binormal vector $B$ is timelike ($\epsilon_3=-1$) or normal vector $N$ is timelike ($\epsilon_2=-1$) and the binormal vector $B$ is spacelike ($\epsilon_3=1$). So we have two following subcases:\\
\textbf{i)}: $\epsilon_2=1$ and $\epsilon_3=-1$.\\
Then the usual transformations between the Walker Frenet frame and the Darboux takes the form
\begin{eqnarray}
Y=\cosh{\theta} N+\sinh{\theta} B\\
U=\sinh{\theta} N+\cosh{\theta} B,
\end{eqnarray}
where $\theta$ is an angle between the vector $Y$ and the vector $N$.\\
Since $\nabla_TT=\kappa N$, we have
\begin{eqnarray}
 \nabla_TT =-\kappa\sinh{\theta} Y+\kappa\cosh{\theta} U.
\end{eqnarray}
Derivating $Y$ along the curve alpha we get 
\begin{eqnarray}
\nabla_TY=-\kappa\sinh{\theta} T+(\theta'+\tau)U
\end{eqnarray}
The same calculus gives 
\begin{eqnarray}
\nabla_TU=-\kappa\cosh{\theta} T+(\theta'+\tau)Y.
\end{eqnarray}
Then the Walker Darboux equation is expressed as
\begin{eqnarray}\label{DF2}
{\left\{ \begin{array}{ccc}
\nabla_TT=-\kappa_g Y+\kappa_n U\\
\nabla_TY=-\kappa_g T+\tau_gU\\
\nabla_TU=-\kappa_n T+\tau_gY,
 \end{array}\right.}
\end{eqnarray}
where $\kappa_g$, $\kappa_n$ and $\tau_g$ are the geodesic curvature, normal curvature and geodesic torsion of $\alpha(s)$ on $S$, respectively. Also, (\ref{DF2}) gives
\begin{eqnarray}\label{eqf}
g^\epsilon_f\left(\nabla_TY,U\right)=\tau_g=\theta'+\tau, \,\,\,\,\,\,\,\,\\
g^\epsilon_f\left(\nabla_TT,Y\right)=\kappa_g=\kappa\sinh{\theta}, \\ \,\,\,\,\,\,\,\,g^\epsilon_f\left(\nabla_TT,U\right)=\kappa_n=\kappa\cosh{\theta}.
\end{eqnarray}
\textbf{ii)}: $\epsilon_2=-1$ and $\epsilon_3=1$.\\
Then the usual transformations between the Walker Frenet frame and the Darboux takes the form
\begin{eqnarray}
Y=\sinh{\theta} N+\cosh{\theta} B\\
U=\cosh{\theta} N+\sinh{\theta} B,
\end{eqnarray}
where $\theta$ is an angle between the vector $Y$ and the vector $N$.\\
Since $\nabla_TT=-\kappa N$, we have 
\begin{eqnarray}
  \nabla_TT =-\kappa\cosh{\theta} Y+\kappa\sinh{\theta} U.
\end{eqnarray}
Derivating $Y$ with respect to $s$ we get 
\begin{eqnarray}
\nabla_TY=-\kappa\cosh{\theta} T+(\theta'-\tau)U
\end{eqnarray}
Derivating $Y$ with respect to $s$ alpha we get
\begin{eqnarray}
\nabla_TU=-\kappa\sinh{\theta} T+(\theta'-\tau)Y.
\end{eqnarray}
Then the Walker Darboux equation is expressed as
\begin{eqnarray}\label{DF3}
{\left\{ \begin{array}{ccc}
\nabla_TT=-\kappa_g Y+\kappa_n U\\
\nabla_TY=-\kappa_g T+\tau_gU\\
\nabla_TU=-\kappa_n T+\tau_gY,
 \end{array}\right.}
\end{eqnarray}
where $\kappa_g$, $\kappa_n$ and $\tau_g$ are the geodesic curvature, normal curvature and geodesic torsion of $\alpha(s)$ on $S$, respectively. Also, (\ref{DF3}) gives
\begin{eqnarray}\label{eqf}
g^\epsilon_f\left(\nabla_TY,U\right)=\tau_g=\theta'-\tau, \,\,\,\,\,\,\,\,\\
g^\epsilon_f\left(\nabla_TT,Y\right)=\kappa_g=\kappa\cosh{\theta}, \\ \,\,\,\,\,\,\,\,g^\epsilon_f\left(\nabla_TT,U\right)=\kappa_n=\kappa\sinh{\theta}.
\end{eqnarray}

\section{Space curves of constant breadth  According to Darboux
Frame in Walker manifold}
In this section, we define space curves of constant breadth in the three
dimensional Walker manifold.
\begin{definition}
  A curve $\alpha: I\rightarrow (M, g_f^\epsilon)$ in the three-dimensional Walker manifold $ (M, g_f^\epsilon)$
is called a curve of constant breadth if there exists a curve $\beta: I\rightarrow M_f$ such
that, at the corresponding points of curves, the parallel tangent vectors of $\alpha$ and $\beta$
 at $\alpha(s)$ and $\beta(s^{\star})$ at $s; s^{\star} \in I$ are opposite directions and the distance  $g_f^\epsilon(\beta-\alpha, \beta-\alpha)$ is  constant. In this case, $(\alpha; \beta)$ is called a  pair curve of constant breadth.
\end{definition}
Let now $(\alpha; \beta)$ be a  pair of unit speed curves of constant breadth and $s, s^{\star}$ be arc-length
of $\alpha$ and $\beta$, respectively. \\
We suppose that the curve $\alpha$ lies on a timelike surface in $M_f$, then it has Darboux frame in addition to Frenet frame. Then we may write the following equation:
\begin{equation}\label{s3eq1}
  \beta(s^{\star}) = \alpha(s) + m_1(s)T(s) + m_2(s)Y(s) + m_3(s)U(s);
\end{equation}
where $ m_i (i=1,2,3)$ are smooth functions of $s$.\\
\subsection{ Case where $\alpha$ is timelike.}
 Differentiating (\ref{s3eq1}) with respect to $s$ and using (\ref{DF1}) we obtain
\begin{eqnarray}\label{s3eq2}
  \frac{d\beta}{ds}&=&\frac{d\beta}{ds^{\star}}\frac{ds^{\star}}{ds}\nonumber\\
  &=&T^{\star}(s^\star)\frac{ds^{\star}}{ds}
  =(1+m'_1+m_2 \kappa_g+m_3\kappa_n)T(s)\nonumber\\
  && +(m'_2 + m_1\kappa_g - m_3\tau_g)Y(s)\nonumber\\
  &&+(m'_3+m_2\tau_g+m_1\kappa_n)U(s),
\end{eqnarray}
where $T^{\star}$ denotes the unit tangent vector of $\beta$.\\
Since $T=-T^*$, from the equations in (\ref{s3eq2}) we have
\begin{equation}\label{s3eq4}
    {\left\{ \begin{array}{ccc}
m'_1&=&-m_2 \kappa_g-m_3\kappa_n-h(s) \\
  m'_2 &=&- m_1\kappa_g + m_3\tau_g  \\
 m'_3&=&-m_2\tau_g-m_1\kappa_n,
\end{array}\right.}
\end{equation}
where $h(s)=\frac{ds^{\star}}{ds}+1$.
We assume that $(\alpha, \beta)$ is a curve pair of constant breadth. Since $\alpha$ is a timelike curve and the vectors $Y$ and $U$ are spacelike vectors, we have
\begin{eqnarray}
  \Vert\beta-\alpha\Vert=-m_1^2+m_2^2+m_3^2=constant,
\end{eqnarray}
which imlplies that
\begin{equation}\label{s3eq5}
  -m_1\frac{dm_1}{ds}+ m_2\frac{dm_2}{ds}+ m_3\frac{dm_3}{ds}=0.
\end{equation}
If we combine (\ref{s3eq4}) and (\ref{s3eq5}), we get
\begin{eqnarray}\label{s3eq6}
  m_1h(s)=0.
\end{eqnarray}
If $\alpha$ and $\beta$ are curves of constant breadth then $m_1=0$ or $h(s)=0$. If $m_1\neq 0$ (that is $h(s)=0$) then $d=m_1T(s)+m_2Y(s)+m_3U(s)$ becomes a constant vector. So $\beta(s^*)$ is a translation of $\alpha$ along the constant vector $d$. Also $h(s)=0$ gives $s^*=-s+c$, where $c$ is constant.\\
Now, we investigate curves of constant breadth for $m_1\neq 0$ or $m_1=0$ in some special case.

\subsubsection{ Case (For geodesic curves)}
Let $\alpha$ be non-straight line geodesic curve on a timelike surface. Then $\kappa_g = \kappa\cos\theta = 0$. As $\kappa\neq 0$, we get $\cos\theta = 0$. So it implies that $\kappa_n = -\kappa$, $\tau_g=-\tau$. From (\ref{s3eq4}), we have
following differential equation system
\begin{equation}\label{s3eq7}
  {\left\{ \begin{array}{ccc}
    m'_1&=& m_3\kappa-h(s)  \\
    m'_2 &=&-m_3\tau \\
   m'_3&=& m_1\kappa+m_2\tau.
\end{array}\right.}
\end{equation}
By using (\ref{s3eq7}), we obtain the following differential equation.
\begin{equation}\label{s3eq9}
\frac{1}{\kappa}\left(\frac{1}{\kappa}(m'_1+h)\right)''+\left[\left(\frac{1}{\kappa}\right)'-\frac{1}{\tau}\left(\frac{\tau}{\kappa}\right)'\right]\left(\frac{1}{\kappa}(m'_1+h)\right)'+\left(\frac{\tau}{\kappa}\right)^2(m'_1+h)+\left(\frac{\tau}{\kappa}\right)'\frac{\kappa}{\tau}m_1-m'_1=0.
\end{equation}
\textbf{Subcase 1:} $m_1\neq 0$ ($h(s)=0)$.\\
If we write $h(s)=0$ in equation (\ref{s3eq9}), we have.
\begin{equation}\label{s3eq10}
\frac{1}{\kappa}\left(\frac{1}{\kappa}m'_1\right)''+\left[\left(\frac{1}{\kappa}\right)'-\frac{1}{\tau}\left(\frac{\tau}{\kappa}\right)'\right]\left(\frac{1}{\kappa}m'_1\right)'+\left[\left(\frac{\tau}{\kappa}\right)^2-1\right]m'_1+\left(\frac{\tau}{\kappa}\right)'\frac{\kappa}{\tau}m_1=0.
\end{equation}
\begin{theorem} Let $\alpha$ be a timelike geodesic curve lying a timelike surface in $M$ and let $(\alpha, \beta)$ be a  pair of unit speed curves of constant breadth. If $m_1$ is a non-zero constant then $\alpha$ is a
general helix in the three dimensional Walker manifold $(M, g_f^\epsilon)$. Also the curve $\beta$ is given as:  
\begin{equation}\label{s3eq11}
  \beta(s^{\star}) = \alpha(s) +m_1T(s)+ m_2Y(s) 
\end{equation}
where $m_2$ is a real constant and $s^*=-s+c$.
\end{theorem}
\begin{proof}
If $m_1$ is non zero constant, then from (\ref{s3eq10}) we obtain that $\left(\frac{\tau}{\kappa}\right)'=0$. So $\alpha$ is a general helix. Also from the first and second equations of (\ref{s3eq7}) we get $m_3=0$ and $m_2$ is a real constant, respectively.
\end{proof}
\begin{theorem}
Let $\alpha$ be a timelike geodesic curve and a general helix lying a timelike surface in $M$. Let $(\alpha, \beta)$ be a pair of unit speed curves of constant breadth. If $m_1$ is not zero, then the curve $\beta$ can be expressed as one of the following cases:
\begin{eqnarray}\label{s3eq15}
  \beta(s^*)=\alpha(s)+m_1T(s)+\frac{1}{c_0}(\Ddot{m}_1-m_1)Y(s)+\Dot{m}_1U(s)
\end{eqnarray}
where 
\begin{itemize} 
    \item[i)]  $m_1=\frac{1}{\sqrt{c_0^2-1}}\left(a_1\sin(\sqrt{c_0^2-1}z)-a_2\cos(\sqrt{c_0^2-1}z)\right)+a_3$, \,\, $c_0^2-1>0$ \\
    \item[ii)] $m_1=\frac{a_1}{2}z^2+a_2z+a_3$, \,\, $c_0^2-1=0$\\
   \item[iii)] $m_1=\frac{1}{\sqrt{1-c_0^2}}\left(a_1\sinh(\sqrt{1-c_0^2}z)+a_2\cosh(\sqrt{1-c_0^2}z)\right)+a_3$, \,\, $c_0^2-1<0$
\end{itemize}
where $z=\int\kappa ds$ and $a_1, a_2, a_3$ are real constants.
\end{theorem}
\begin{proof}
Let us consider that $\alpha$ is timelike geodesic curve and a general helix in Wlaker 3-manifold. Then we have $\frac{\tau}{\kappa}=c_0=constant$. From (\ref{s3eq10}), we have 
\begin{equation}\label{s3eq13}
\left(\frac{1}{\kappa}\left(\frac{1}{\kappa}m'_1\right)'\right)'+\left(c_0^2-1\right)m'_1=0.
\end{equation}
By means of changing of the independant variable $s$ with $z=\int\kappa ds$, from (\ref{s3eq13}) we obtain
$$m'_1=\frac{dm_1}{ds}=\frac{dm_1}{dz}\frac{dz}{ds}=\Dot{m}_1\kappa.$$
\begin{eqnarray}\label{s3eq14}
  \dddot{m}_1+(c_0^2-1)\Dot{m}_1=0.
\end{eqnarray}
If we solve this equation we get 
\begin{displaymath}
m_1={\left\{ \begin{array}{ccc}
\frac{1}{\sqrt{c_0^2-1}}\left(a_1\sin(\sqrt{c_0^2-1}z)-a_2\cos(\sqrt{c_0^2-1}z)\right)+a_3, \,\,\text{if} \,\,c_0^2-1>0 \\
\frac{a_1}{2}z^2+a_2z+a_2, \,\,\text{if} \,\,c_0^2-1=0\\
\frac{1}{\sqrt{1-c_0^2}}\left(a_1\sinh(\sqrt{1-c_0^2}z)+a_2\cosh(\sqrt{1-c_0^2}z)\right)+a_3, \,\,\text{if} \,\,c_0^2-1<0.
\end{array}\right.}
\end{displaymath}
From (\ref{s3eq7}) we obtain $m_3=\Dot{m}_1$ and $m_2=\frac{1}{c_0}(\Ddot{m}_1-m_1)$. 
\end{proof}
\textbf{Subcase 2: $m_1=0$.}\\
If we take $m_1=0$ in the equation  (\ref{s3eq7}), we get 
\begin{equation}\label{s3eq16}
  {\left\{ \begin{array}{ccc}
    h(s)&=& m_3\kappa   \\
    m'_2 &=&-m_3\tau\\
   m'_3&=& m_2\tau.
\end{array}\right.}
\end{equation}
Since $m_3=\frac{h}{\kappa}$, $m_2=\frac{1}{\tau}m'_3=\frac{1}{\tau}\left(\frac{h}{\kappa}\right)'$, we get 
\begin{eqnarray}\label{s3eq17}
  \left[\frac{1}{\tau}\left(\frac{h}{\kappa}\right)'\right]'+\left(\frac{h}{\kappa}\right)\tau=0.
\end{eqnarray}
If we put $y=\frac{h}{\kappa}$, the equation (\ref{s3eq17}) becomes
\begin{eqnarray}\label{s3eq18}
  y''-\frac{\tau'}{\tau}y'+\tau^2y=0.
\end{eqnarray}
For solving the equation (\ref{s3eq18}), we put the new variable $\frac{dw}{ds}=\tau.$ Then
\begin{eqnarray}\label{s3eq19}
  {\left\{ \begin{array}{cc}
y'=\frac{dy}{dw}\frac{dw}{ds}=\Dot{y}\tau  \\
y''=\frac{d^2y}{dw^2}\tau^2+\frac{dy}{dw}\tau'
\end{array}\right.}
\end{eqnarray}
If we put the equation (\ref{s3eq19}) in the equation (\ref{s3eq18})
we obtain
\begin{equation}\label{s3eq20}
\frac{d^2y}{dw^2}+y=0.
\end{equation}
and the solution of (\ref{s3eq20}) is $y=b_1\cos w+b_2\sin w$. Then we have 
\begin{eqnarray}\label{s3eq21}
  h(s)=\kappa\left[b_1\cos\left(\int\tau ds\right)+b_2\sin\left(\int\tau ds\right)\right]\\
  m_2=\frac{h}{\kappa}=b_1\cos\left(\int\tau ds\right)+b_2\sin\left(\int\tau ds\right)\\
  m_3=\frac{1}{\tau}\left(\frac{h}{\kappa}\right)'=-b_1\sin\left(\int\tau ds\right)+b_2\cos\left(\int\tau ds\right).
\end{eqnarray}
So we give the following theorem
\begin{theorem}
Let $\left(\alpha, \beta\right)$ be a pair of constant breadth curve in $(M, g_f)$ where $\alpha$ is
a timelike geodesic curve lying in a timelike surface in $M$. If $m_1=0$, then the curve $\beta$ is given by

\begin{displaymath}
  \beta(s^*) = \alpha(s) + \left[b_1\cos\left(\int\tau ds\right)+b_2\sin\left(\int\tau ds\right)\right]Y(s)+\left[-b_1\sin\left(\int\tau ds\right)\\
  +b_2\cos\left(\int\tau ds\right)\right]U(s).
\end{displaymath}
\end{theorem}
\subsubsection{Case (For asymptotic lines)}
Let $\alpha$ be non-straight line asymptotic line on a timelike surface. Then $\kappa_n= -\kappa\sin\theta = 0$. As $\kappa\neq 0$, we get $\sin\theta = 0$. So it implies that $\kappa_g = \kappa$, $\tau_g=-\tau$. From  (\ref{s3eq4}), we have
following differential equation system
\begin{equation}\label{s3eq22}
  {\left\{ \begin{array}{ccc}
    m'_1 &=& -m_2\kappa-h(s)  \\
    m'_2 &=&-m_1\kappa-m_3\tau \\
    m'_3 &=& m_2\tau.
\end{array}\right.}
\end{equation}
By using (\ref{s3eq22}), we get 
\begin{equation}\label{s3eq23}
\frac{1}{\kappa}\left(\frac{1}{\kappa}(m'_1+h)\right)''+\left[\left(\frac{1}{\kappa}\right)'-\frac{1}{\tau}\left(\frac{\tau}{\kappa}\right)'\right]\left(\frac{1}{\kappa}(m'_1+h)\right)'+\left(\frac{\tau}{\kappa}\right)^2(m'_1+h)+\left(\frac{\tau}{\kappa}\right)'\frac{\kappa}{\tau}m_1-m'_1=0.
\end{equation}

\textbf{Subcase 1:} $m_1\neq 0$ ($h(s)=0)$.\\
If we take as $h(s)=0$ in equation (\ref{s3eq23}), we get following differential equation
\begin{equation}\label{s3eq24}
\frac{1}{\kappa}\left(\frac{1}{\kappa}m'_1\right)''+\left[\left(\frac{1}{\kappa}\right)'-\frac{1}{\tau}\left(\frac{\tau}{\kappa}\right)'\right]\left(\frac{1}{\kappa}m'_1\right)'+\left[\left(\frac{\tau}{\kappa}\right)^2-1\right]m'_1+\left(\frac{\tau}{\kappa}\right)'\frac{\kappa}{\tau}m_1=0.
\end{equation}
\begin{theorem} Let $\alpha$ be a timelike asymptotic line lying a timelike surface in $M$. 
Let $(\alpha, \beta)$ be a  pair of unit speed curves of constant breadth. If $m_1$ is non-zero constant then $\alpha$ is a
general helix in the three dimensional Walker manifold $(M, g_f^\epsilon)$. Also the curve $\beta$ is given as:  
\begin{equation}\label{s3eq25}
  \beta(s^{\star}) = \alpha(s) +m_1T(s) + m_3U(s)
\end{equation}
where $m_3$ is a real constant and $s^*=-s+c$.
\end{theorem}
\begin{proof}
If $m_1$ is non zero constant, then from (\ref{s3eq24}) we obtain that $\left(\frac{\tau}{\kappa}\right)'=0$. So $\alpha$ is a general helix. Also from the first and third equation of (\ref{s3eq22}) we get $m_2=0$ and $m_3$ is a real constant, respectively.
\end{proof}
\begin{theorem}
Let $\alpha$ be a timelike asymptotic line lying in a timelike surface in $M$. Let $(\alpha, \beta)$ be a pair of unit speed curves of constant breadth. If $m_1$ is not zero, then the curve $\beta$ can be expressed as one of the following cases:
\begin{eqnarray}\label{s3eq29}
  \beta(s^*)=\alpha(s)+m_1T(s)-\Dot{m}_1Y(s)+\frac{1}{c_0}(\Ddot{m}_1-m_1)U(s),
\end{eqnarray}
where 
\begin{itemize}
    \item[i)]  $m_1=\frac{1}{\sqrt{c_0^2-1}}\left(a_1\sin(\sqrt{c_0^2-1}z)-a_2\cos(\sqrt{c_0^2-1}z)\right)+a_3$,\,\, $c_0^2-1>0$ \\
   \item[ii)] $m_1=\frac{a_1}{2}z^2+a_2z+a_3$,\,\,   $c_0^2-1=0$\\
   \item[iii)] $m_1=\frac{1}{\sqrt{1-c_0^2}}\left(a_1\sinh(\sqrt{1-c_0^2}z)+a_2\cosh(\sqrt{1-c_0^2}z)\right)+a_3$,\,\, $c_0^2-1<0$
\end{itemize}
where $z=\int\kappa ds$ and $a_1, a_2, a_3$ are constants.
\end{theorem}
\begin{proof}
The proof of Theorem (4.6) is done similarly to the proof of Theorem (4.3)
\end{proof}

\textbf{Subcase 2: $m_1=0$.}\\
If we take as $m_1=0$ in (\ref{s3eq22}) we get following differential equation system
\begin{equation}\label{s3eq30}
  {\left\{ \begin{array}{ccc}
    h(s)&=& -m_2\kappa   \\
    m'_2 &=&-m_3\tau\\
   m'_3&=& m_2\tau.
\end{array}\right.}
\end{equation}
Then we give the following theorem.
\begin{theorem}
Let $\left(\alpha; \beta\right)$ be a curve pair of constant breadth in $(M, g_f)$ where $\alpha$ is
a timelike asymptotic curve lying in a timelike surface in $M$. If $m_1=0$, then the curve $\beta$ is given by
\begin{displaymath}
  \beta(s^*) = \alpha(s) + \left[-b_1\cos\left(\int\tau ds\right)-b_2\sin\left(\int\tau ds\right)\right]Y(s)+\left[-b_1\sin\left(\int\tau ds\right)\\
  +b_2\cos\left(\int\tau ds\right)\right]U(s).
\end{displaymath}
\end{theorem}
\begin{proof}
The proof of Theorem (4.7) is done similarly to the proof of Theorem (4.4).
\end{proof}
\subsubsection{Case (For Principal line)}
 We suppose that $\alpha$ is a non-planar timelike principal line. Then we have $\tau_g=0$. Then it follows that $\tau=\theta'$. By using (\ref{s3eq4}), we have the following differential equation system
\begin{equation}\label{s3eq35}
  {\left\{ \begin{array}{ccc}
    m'_1&=& m_3\kappa\sin\theta-m_2\kappa\cos\theta-h(s)  \\
    m'_2 &=&-m_1\kappa\cos\theta \\
   m'_3&=& m_1\kappa\sin\theta.
\end{array}\right.}
\end{equation}
By mean of changing of the independant variable $s$ with $\theta=\int\tau ds$, we get
\begin{equation}\label{s3eq37}
  {\left\{ \begin{array}{ccc}
    \Dot{m}_1&=& \phi(m_3\sin\theta-m_2\cos\theta)-g(\theta) \\
    \Dot{m}_2 &=&-m_1\phi\cos\theta \\
   \Dot{m}_3&=& m_1\phi\sin\theta.
\end{array}\right.}
\end{equation}
where $g(\theta)=(-\frac{ds}{d\theta}-\frac{ds^*}{d\theta})$ and $\phi=\frac{\kappa}{\tau}$. In here we denote the derivative with respect to $\theta$ with ".".
From the equations in (\ref{s3eq37}) we have
\begin{eqnarray}\label{s3eq39}
  \dddot{m}_1+\ddot{g}-\frac{d}{d\theta}\left(\frac{\dot{\phi}}{\phi}(\dot{m}_1+g)\right)-\frac{d}{d\theta}(\phi^2m_1)+(\dot{m}_1+g)\nonumber\\-\dot{\phi}\left(-\sin\theta\int m_1\phi\cos\theta d\theta+\cos\theta\int m_1\phi\sin\theta d\theta\right)=0.
\end{eqnarray}

\textbf{Subcase 1:} $m_1\neq 0$ ($h(s)=0)$.\\
In this case, we give the following theorem:
\begin{theorem}
Let $(\alpha, \beta)$ be a pair curves of constant breadth in $(M, g_f\epsilon)$. Let $\alpha$ be a non-planar timelike principal line and a general helix then $\beta$ is given by one of the following cases:
\begin{eqnarray}
  \beta(s^*)=\alpha(s)+m_1T(s)-c\int m_1\cos\theta d\theta Y(s)+c\int m_1\sin\theta d\theta U(s),
\end{eqnarray}
where 
\begin{itemize} 
    \item[i)]  $m_1=\frac{1}{\sqrt{1-c^2}}\left(a_1\sin(\sqrt{1-c^2}\theta)-a_2\cos(\sqrt{1-c^2}\theta)\right)+a_3$, \,\, $1-c^2>0$ \\
    \item[ii)] $m_1=\frac{a_1}{2}\theta^2+a_2\theta+a_3$, \,\,   $c^2-1=0$\\
   \item[iii)] $m_1=\frac{1}{\sqrt{c^2-1}}\left(a_1\sinh(\sqrt{c^2-1}\theta)+a_2\cosh(\sqrt{c^2-1}\theta)\right)+a_3$, \,\, $1-c^2<0$
\end{itemize}
\end{theorem}
\begin{proof}
If $h(s)=0$ then $g(\theta)=0$ and from (\ref{s3eq39}) we have
\begin{eqnarray}\label{s3eq40}
  \dddot{m}_1-\frac{d}{d\theta}\left(\frac{\dot{\phi}}{\phi}\dot{m}_1\right)-\frac{d}{d\theta}(\phi^2m_1)+\dot{m}_1-\dot{\phi}\left(-\sin\theta\int m_1\phi\cos\theta d\theta+\cos\theta\int m_1\phi\sin\theta d\theta\right)=0.
\end{eqnarray}
If $\alpha$ is helix curve then $\phi=\frac{\kappa}{\tau}=c=constant$. From (\ref{s3eq40}) we have
\begin{eqnarray}\label{s3eq41}
  \dddot{m}_1+(1-c^2)\dot{m}_1=0.
\end{eqnarray}
Then the solution is 
\begin{displaymath}
m_1={\left\{ \begin{array}{ccc}
\frac{1}{\sqrt{1-c^2}}\left(a_1\sin(\sqrt{1-c^2}\theta)-a_2\cos(\sqrt{1-c^2}\theta)\right)+a_3, \,\,\text{if} \,\,1-c^2>0 \\
\frac{a_1}{2}\theta^2+a_2\theta+a_3, \,\,\,\,\,\,\,\text{if} \,\,\,\,\,\,\,1-c^2=0\\
\frac{1}{\sqrt{c^2-1}}\left(a_1\sinh(\sqrt{c^2-1}\theta)+a_2\cosh(\sqrt{c^2-1}\theta)\right)+a_3, \,\,\text{if} \,\,1-c^2<0,
\end{array}\right.}
\end{displaymath}
where $\theta=\int\tau d\theta$.
\end{proof}
\textbf{Subcase 2:} $m_1= 0$.\\
The case where $m_1=0$, we have the following the following theorem:
\begin{theorem}
Let $(\alpha, \beta)$ be a pair curves of constant breadth in $(M, g_f\epsilon)$. Let $\alpha$ be a non-planar timelike principal line. If $m_1=0$ then $\alpha$ is general helix. The curve $\beta$ is expressed as
\begin{eqnarray}
  \beta(s^*)=\alpha(s)+c_2Y(s)+c_3U(s),
\end{eqnarray}
where $c_2$ and $c_3$ are constants.
\end{theorem}
\begin{proof}
From (\ref{s3eq39}) we have
\begin{eqnarray}\label{s3eq41}
 \ddot{g}-\frac{d}{d\theta}\left(\frac{\dot{\phi}}{\phi}g\right)+g=0.
\end{eqnarray}
On the other hand, from (\ref{s3eq35}) we have $m_2=c_2=constant\neq 0$, $m_3=c_3=constant\neq 0$ and from (\ref{s3eq37})
\begin{eqnarray}\label{s3eq42}
  g=\phi(-c_2\cos\theta+c_3\sin\theta).
\end{eqnarray}
By considering (\ref{s3eq41}) and (\ref{s3eq42}) with together, we get
\begin{eqnarray}\label{s3eq43}
  \dot{\phi}(c_2\sin\theta+c_3\cos\theta)=0.
\end{eqnarray}
Then we have $\dot{\phi}=0$ or $c_2\sin\theta+c_3\cos\theta=0$. If $c_2\sin\theta+c_3\cos\theta=0$ then we have that $\theta$ is a constant. So $\alpha$ becomes a planar curve. It is a contridiction. So $\dot{\phi}=0$. Then we obtain that ${\phi=\frac{\kappa}{\tau}}$ is a constant. Thus $\alpha$ is a general helix.
\end{proof}

\subsection{ Case where $\alpha$ is spacelike and $\epsilon_2=1$ and $\epsilon_3=-1$.}
Here we suppose that the curve $\alpha$ is spacelike and lying on a timelike surface in $M_f$.\\
Differentiating (\ref{s3eq1}) with respect to $s$ and using (\ref{DF2}) we obtain
\begin{eqnarray}\label{Case2.1}
  \frac{d\beta}{ds}&=&\frac{d\beta}{ds^{\star}}\frac{ds^{\star}}{ds}\nonumber\\
  &=&T^{\star}\frac{ds^{\star}}{ds}=(1+m'_1-m_2 \kappa_g-m_3\kappa_n)T\nonumber\\
  && +(m'_2 -m_1\kappa_g +m_3\tau_g)Y\nonumber\\
  &&+(m'_3+m_2\tau_g+m_1\kappa_n)U,
\end{eqnarray}
where $T^{\star}$ denotes the tangent vector of $\beta$.\\
Since $T=-T^*$, from the equation in (\ref{s3eq2}) we have
\begin{equation}\label{case2.2}
    {\left\{ \begin{array}{ccc}
m'_1&=&m_2 \kappa_g+m_3\kappa_n -h(s) \\
  m'_2 &=& m_1\kappa_g - m_3\tau_g  \\
 m'_3&=&-m_2\tau_g-m_1\kappa_n,
\end{array}\right.}
\end{equation}
where $h(s)=\frac{ds^*}{ds}+1$.\\
Since $\alpha$ is spacelike and $\epsilon_2=1$ and$\epsilon_3=-1$, then, if we assume that $(\alpha, \beta)$ is a curve pair of constant breadth, we have
\begin{eqnarray}
  \Vert\beta-\alpha\Vert=m_1^2+m_2^2-m_3^2=constant,
\end{eqnarray}
which imlplies that
\begin{equation}\label{case2.3}
  m_1\frac{dm_1}{ds}+ m_2\frac{dm_2}{ds}- m_3\frac{dm_3}{ds}=0.
\end{equation}
If we combine (\ref{case2.2}) and (\ref{case2.3}) we get
\begin{eqnarray}\label{case2.4}
  m_1(2m'_1+h(s))=0.
\end{eqnarray}
If $\alpha$ and $\beta$ are curves of constant breadth then $m_1=0$ or $2m'_1-h(s)=0$.\\
Now we investigate the case where $\alpha$ is geodesic curve or principal line curve because $\kappa_n\neq 0$.
\subsubsection{ Case (For geodesic curves)}
Let $\alpha$ be non-straight line geodesic curve on a timelike surface. Then $\kappa_g = \kappa\sinh\theta = 0$. As $\kappa\neq 0$, we get $\sinh\theta = 0$. So it implies that $\kappa_n = \kappa$, $\tau_g=\tau$. From (\ref{case2.2}), we have the
following differential equation system
\begin{equation}\label{G1}
  {\left\{ \begin{array}{ccc}
    m'_1&=& m_3\kappa-h(s)  \\
    m'_2 &=&-m_3\tau \\
   m'_3&=& -m_1\kappa-m_2\tau.
\end{array}\right.}
\end{equation}
From (\ref{G1}) we have 
\begin{equation}\label{G2}
  {\left\{ \begin{array}{ccc}
    m_3&=& \frac{1}{\kappa}(m'_1+h)  \\
    m'_2 &=&-\frac{\tau}{\kappa}(m'_1+h)\\
   m_2&=& -\frac{1}{\tau}\left((\frac{1}{\kappa}(m'_1+h))'+m_1\kappa\right).
\end{array}\right.}
\end{equation}
Differentiating the third equation of (\ref{G1}) with respect to $s$ and using the first, the second and the third equations of (\ref{G2}), we obtain the following equation:
\begin{equation}\label{G3}
\frac{1}{\kappa}\left(\frac{1}{\kappa}(m'_1+h)\right)''+\left[\left(\frac{1}{\kappa}\right)'-\frac{1}{\tau}\left(\frac{\tau}{\kappa}\right)'\right]\left(\frac{1}{\kappa}(m'_1+h)\right)'-\left(\frac{\tau}{\kappa}\right)^2(m'_1+h)-\left(\frac{\tau}{\kappa}\right)'\frac{\kappa}{\tau}m_1+m'_1=0.
\end{equation}

\textbf{Subcase 1:} $m_1\neq 0$  ($h(s)=-2m'_1)$.\\
The equation (\ref{G3}) becomes
\begin{equation}\label{G4}
\frac{1}{\kappa}\left(\frac{1}{\kappa}m'_1\right)''+\left[\left(\frac{1}{\kappa}\right)'-\frac{1}{\tau}\left(\frac{\tau}{\kappa}\right)'\right]\left(\frac{1}{\kappa}m'_1\right)'-\left[\left(\frac{\tau}{\kappa}\right)^2+1\right]m'_1+\left(\frac{\tau}{\kappa}\right)'\frac{\kappa}{\tau}m_1=0.
\end{equation}

\begin{theorem} Let $\alpha$ be a geodesic curve. 
Let $(\alpha; \beta)$ be a  pair of unit speed curves of constant breadth where $\alpha$ is spacelike ($\epsilon_2=1$, $\epsilon_3=-1$) and lying in a timelike surface in $M_f$. If $m_1$ is non-zero constant then $m_3=0$ and $\alpha$ is a
general helix in the three dimensional Walker manifold $(M, g_f^\epsilon)$. Also the curve $\beta$ is given as:  
\begin{equation}\label{G5}
  \beta(s^{\star}) = \alpha(s) +m_1T + cY 
\end{equation}
where $c$ is a real constant and $s^*=-s+c$.
\end{theorem}
\begin{proof}
If $m_1$ is non zero constant, then from (\ref{G4}) we obtain that $\left(\frac{\tau}{\kappa}\right)'=0$. So $\alpha$ is a general helix. Also from the second and third equation of (\ref{G1}) we get $m_3=0$ because $h=0$ and $m_2$ is a real constant.
\end{proof}

\begin{theorem}
Let $\alpha$ be a geodesic curve. Let $(\alpha, \beta)$ be a pair of unit speed curves of constant breadth where $\alpha$ is spacelike curve ($\epsilon_2=1$, $\epsilon_3=-1$) and lying in a timelike surface $M_f$. If $m_1$ is not zero, then the curve $\beta$ can be expressed as one of the following cases:
\begin{eqnarray}\label{s3eq15}
  \beta(s^*)=\alpha(s)+m_1T+\frac{1}{c_0}(\Ddot{m}_1-m_1)Y+\Dot{m}_1U,
\end{eqnarray}
where 
  $m_1=\frac{1}{\sqrt{1+c_0^2}}\left(a_1e^{\sqrt{1+c_0^2}\theta}-a_2e^{-\sqrt{1+c_0^2}\theta}\right)$,
$m_3=-\dot{m}_1$ and $m_2=\frac{1}{c_0}(\ddot{m}_1-m_1)$.
\end{theorem}
\begin{proof}
Let us consider that $\alpha$ is a general helix in Wlaker 3-manifold. Then we have $\frac{\tau}{\kappa}=c_0=constant$. From (\ref{G4}), we have 

\begin{equation}\label{G6}
\left(\frac{1}{\kappa}\left(\frac{1}{\kappa}m'_1\right)'\right)'-\left(c_0^2+1\right)m'_1=0.
\end{equation}
By means of changing of the independant variable $s$ with $z=\int\kappa ds$, we obtain
$$m'_1=\frac{dm_1}{ds}=\frac{dm_1}{dz}\frac{dz}{ds}=\Dot{m}_1\kappa.$$
From (\ref{G6}), we get 
\begin{eqnarray}\label{s3eq14}
  \dddot{m}_1-(c_0^2+1)\Dot{m}_1=0.
\end{eqnarray}
If we solve this equation we get
\begin{eqnarray}\label{G7}
  m_1=\frac{1}{\sqrt{1+c_0^2}}\left(a_1e^{\sqrt{1+c_0^2}\theta}-a_2e^{-\sqrt{1+c_0^2}\theta}\right).
\end{eqnarray}
From (\ref{G2}) we have $m_3=-\dot{m}_1$ and $m_2=\frac{1}{c_0}(\ddot{m}_1-m_1)$.
\end{proof}
\textbf{Subcase 2:} $m_1= 0$.\\
\begin{theorem}
Let $(\alpha, \beta)$ be a pair curves of constant breadth in $(M, g_f\epsilon)$. Let $\alpha$ be a geodesic spacelike curve ($\epsilon_2=1$, $\epsilon_3=-1$) and lying in a timelike surface on $M_f$. If $m_1=0$ then the curve $\beta$ is expressed as
\begin{eqnarray}
  \beta(s^*)=\alpha(s)+cY,
\end{eqnarray}
where $c$ is a constant real.
\end{theorem}
\begin{proof}
If $m'_1=0$ then $h=0$ and from (\ref{G1}) we have $m_3=0$ and $m_2=constant$.
\end{proof}

\subsubsection{ Case (For Principal line)}

If $\alpha$ is principal line, then $\tau_g=0$ and $\tau=-\theta'$. From (\ref{case2.2})
\begin{equation}\label{case2.5}
    {\left\{ \begin{array}{ccc}
m'_1&=&m_2 \kappa\sinh\theta+m_3\kappa\cosh\theta -h(s) \\
  m'_2 &=& m_1\kappa\sinh\theta \\
 m'_3&=&-m_1\kappa\cosh\theta,
\end{array}\right.}
\end{equation}
By mean of changing of the independant variable $s$ with $\theta=\int\tau ds$, we get
\begin{equation}\label{case2.6}
  {\left\{ \begin{array}{ccc}
    \Dot{m}_1&=& m_3\frac{\kappa}{\tau}\cosh\theta+m_2\frac{\kappa}{\tau}\sinh\theta-\frac{h(s)}{\tau(s)}  \\
    \Dot{m}_2 &=&m_1\frac{\kappa}{\tau}\sinh\theta \\
   \Dot{m}_3&=&- m_1\frac{\kappa}{\tau}\cosh\theta.
\end{array}\right.}
\end{equation}
Denoted by $\frac{h(s)}{\tau(s)}=g(\theta)$ and $\frac{\kappa}{\tau}=\phi$, we have 
\begin{equation}\label{case2.7}
  {\left\{ \begin{array}{ccc}
    \Dot{m}_1&=& \phi(m_3\cosh\theta+m_2\sinh\theta)-g(\theta) \\
    \Dot{m}_2 &=&m_1\phi\sinh\theta \\
   \Dot{m}_3&=& -m_1\phi\cosh\theta.
\end{array}\right.}
\end{equation}
From the equations in (\ref{case2.7}) we have
\begin{equation}\label{case2.8}
  {\left\{ \begin{array}{ccc}
    \frac{1}{\phi}(\Dot{m}_1+g)&=& m_3\cosh\theta+m_2\sinh\theta \\
    \Dot{m}_2\sinh\theta+\Dot{m}_3\cosh\theta&=& -m_1\phi \\
   \Dot{m}_2\cosh\theta&=& -m_3\sinh\theta.
\end{array}\right.}
\end{equation}
Differentiating the first equation in (\ref{case2.7}), we get
\begin{eqnarray}\label{case2.9}
  \dddot{m}_1+\ddot{g}-\frac{d}{d\theta}\left(\frac{\dot{\phi}}{\phi}(\dot{m}_1+g)\right)+\frac{d}{d\theta}(\phi^2m_1)-(\dot{m}_1+g)\nonumber\\-\dot{\phi}\left(\cosh\theta\int m_1\phi\sinh\theta d\theta-\sinh\theta\int m_1\phi\cosh\theta d\theta\right)=0.
\end{eqnarray}

\textbf{Subcase 1:} $m_1\neq 0$ ($m'_1=-\frac{h}{2})$.\\
If $m'_1=-\frac{h}{2}$ then $\dot{m}_1=-\frac{g}{2}$. From (\ref{case2.9}) we obtain 
\begin{eqnarray}\label{case2.10}
  -\dddot{m}_1+\frac{d}{d\theta}\left(\frac{\dot{\phi}}{\phi}\dot{m}_1\right)+\frac{d}{d\theta}(\phi^2m_1)+\dot{m}_1-\dot{\phi}\left(\cosh\theta\int m_1\phi\sinh\theta d\theta-\sinh\theta\int m_1\phi\cosh\theta d\theta\right)=0.
\end{eqnarray}

\begin{theorem}
Let $(\alpha, \beta)$ be a pair curves of constant breadth in $(M, g_f\epsilon)$. Let $\alpha$ be principal line and a general helix then $\beta$ is given by 
\begin{eqnarray}
  \beta(s^*)=\alpha(s)+m_1T+m_2Y+m_3U,
\end{eqnarray}
where 
\begin{eqnarray*}
  m_1=\frac{1}{\sqrt{1+c^2}}\left(a_1e^{\sqrt{1+c^2}\theta}-a_2e^{-\sqrt{1+c^2}\theta}\right),
\end{eqnarray*}
$m_2=c\int m_1\sinh\theta d\theta$ and $m_3=-c\int m_1\cosh\theta d\theta$.
\end{theorem}
\begin{proof}
If $\alpha$ is helix curve then $\phi=\frac{\kappa}{\tau}=c=constant$. From (\ref{case2.10}) we have
\begin{eqnarray}\label{case2.11}
  \dddot{m}_1-(1+c^2)\dot{m}_1=0.
\end{eqnarray}
\begin{eqnarray}\label{case2.12}
  m_1=\frac{1}{\sqrt{1+c^2}}\left(a_1e^{\sqrt{1+c^2}\theta}-a_2e^{-\sqrt{1+c^2}\theta}\right).
\end{eqnarray}
\end{proof}
\textbf{Subcase 2:} $m_1= 0$.\\
From the equations in (\ref{case2.2}) we have 
$m_2=c_2=constant\neq 0$, $m_3=c_3=constant\neq 0$. The first equation in (\ref{case2.2}) gives
\begin{eqnarray}
  \tanh\theta=-\frac{c_2}{c_3}.
\end{eqnarray}
Then $\theta$ is a constant and we have $\tau=0$.
\begin{theorem}
Let $(\alpha, \beta)$ be a pair curves of constant breadth in $(M, g_f\epsilon)$. Let $\alpha$ be principal line. If $m_1=0$ then $\alpha$ is planar curve. The curve $\beta$ is expressed as
\begin{eqnarray}
  \beta(s^*)=\alpha(s)+c_2Y+c_3U,
\end{eqnarray}
where $c_2$ and $c_3$ are constants.
\end{theorem}

\subsection{ Case where $\alpha$ is spacelike and $\epsilon_2=-1$ and $\epsilon_3=1$.}
Let $\alpha$ be a spacelike with $\epsilon_2=-1$ and $\epsilon_3=1$ lying on a timelike surface in $M_f$.\\
Differentiating (\ref{s3eq1}) with respect to $s$ and using (\ref{DF3}) we obtain

\begin{equation}\label{case3.1}
    {\left\{ \begin{array}{ccc}
m'_1&=&m_2 \kappa_g+m_3\kappa_n -h(s) \\
  m'_2 &=& m_1\kappa_g - m_3\tau_g  \\
 m'_3&=&-m_2\tau_g-m_1\kappa_n,
\end{array}\right.}
\end{equation}
where $h(s)=\frac{ds^*}{ds}+1$.\\
Since $\alpha$ is spacelike and $\epsilon_2=-1$ and$\epsilon_3=1$, then, if we assume that $(\alpha, \beta)$ is a curve pair of constant breadth, we have
\begin{eqnarray}
  \Vert\beta-\alpha\Vert=m_1^2-m_2^2+m_3^2=constant,
\end{eqnarray}
which imlplies that
\begin{equation}\label{case3.2}
  m_1\frac{dm_1}{ds}+ m_2\frac{dm_2}{ds}- m_3\frac{dm_3}{ds}=0.
\end{equation}
If we combine (\ref{case3.1}) and (\ref{case3.2}) we get
\begin{eqnarray}\label{case3.3}
  m_1h(s)=0.
\end{eqnarray}
If $\alpha$ and $\beta$ are curves of constant breadth then $m_1=0$ or $h(s)=0$.
If $m_1\neq 0$ (that is $h(s)=0$) then $d=m_1T+m_2Y+m_3U$ becomes a constant vector because $d'=0$. So $\beta(s^*)$ is a translation of $\alpha$ along the constant vector $d$. Also $h(s)=0$ gives $s^*=-s+c$, where $c$ is constant.\\
Since $\kappa_g\neq 0$, here we investigate curves of constant breadth for $m_1\neq 0$ or $m_1=0$ in some special case (asymptotic line or principal line).

\subsubsection{ Case (For Asymptotic line)}
Let $\alpha$ be non-straight line asymptotic line on a timelike surface. Then $\kappa_n= \kappa\sinh\theta = 0$. As $\kappa\neq 0$, we get $\cosh\theta = 0$. So it implies that $\kappa_g = \kappa$, $\tau_g=-\tau$. From  (\ref{case3.1}), we have
following differential equation system
\begin{equation}\label{A1}
  {\left\{ \begin{array}{ccc}
    m'_1 &=& m_2\kappa-h(s)  \\
    m'_2 &=&m_1\kappa+m_3\tau \\
    m'_3 &=& -m_2\tau.
\end{array}\right.}
\end{equation}
By differentiating the second equation in (\ref{A1}) with respect to $s$ and using the first and third equations in (\ref{A1}), we get 
\begin{equation}\label{A3}
\frac{1}{\kappa}\left(\frac{1}{\kappa}(m'_1+h)\right)''+\left[\left(\frac{1}{\kappa}\right)'-\frac{1}{\tau}\left(\frac{\tau}{\kappa}\right)'\right]\left(\frac{1}{\kappa}(m'_1+h)\right)'-\left(\frac{\tau}{\kappa}\right)^2(m'_1+h)+\left(\frac{\tau}{\kappa}\right)'\frac{\kappa}{\tau}m_1-m'_1=0.
\end{equation}

\textbf{Subcase 1:} $m_1\neq 0$ ($h(s)=0)$.\\
The equation (\ref{A3}) is given by 
\begin{equation}\label{A4}
\frac{1}{\kappa}\left(\frac{1}{\kappa}m'_1\right)''+\left[\left(\frac{1}{\kappa}\right)'-\frac{1}{\tau}\left(\frac{\tau}{\kappa}\right)'\right]\left(\frac{1}{\kappa}m'_1\right)'-\left[\left(\frac{\tau}{\kappa}\right)^2+1\right]m'_1+\left(\frac{\tau}{\kappa}\right)'\frac{\kappa}{\tau}m_1=0.
\end{equation}

\begin{theorem} Let $\alpha$ be a asymptotic curve. 
Let $(\alpha; \beta)$ be a  pair of unit speed curves of constant breadth where $\alpha$ is spacelike (with $\epsilon_2=-1$ and $\epsilon_3=1$) lying in a timelike surface in $M_f$. If $m_1$ is non-zero constant then $m_2=0$ and $\alpha$ is a
general helix in the three dimensional Walker manifold $(M, g_f^\epsilon)$. Also the curve $\beta$ is given as:  
\begin{equation}\label{s3eqA5}
  \beta(s^{\star}) = \alpha(s) +m_1T + m_3U 
\end{equation}
where $m_3$ is a real constant and $s^*=-s+c$.
\end{theorem}
\begin{proof}
If $m_1$ is non zero constant, then from (\ref{A4}) we obtain that $\left(\frac{\tau}{\kappa}\right)'=0$. So $\alpha$ is a general helix. Also from the first and third equation of (\ref{A1}) we get $m_2=0$ and $m_3$ is a real constant.
\end{proof}

\begin{theorem}
Let $\alpha$ be a asymptotic line. Let $(\alpha, \beta)$ be a pair of unit speed curves of constant breadth where $\alpha$ is timelike curve and lying in a timelike surface $M_f$. If $m_1$ is not zero, then the curve $\beta$ can be expressed as one of the following cases:
\begin{eqnarray}\label{A9}
  \beta(s^*)=\alpha(s)+m_1T+\Dot{m}_1Y+\frac{1}{c_0}(\Ddot{m}_1+m_1)U,
\end{eqnarray}
where
$$m_1=\frac{1}{\sqrt{c_0^2+1}}\left(a_1e^{\sqrt{c_0^2+1}z}-a_2e^{\sqrt{c_0^2+1}z}\right).$$
\end{theorem}
\begin{proof}
Let us consider that $\alpha$ is a general helix in Walker 3-manifold. Then we have $\frac{\tau}{\kappa}=c_0=constant$. From (\ref{A4}), we have 
\begin{equation}\label{A6}
\left(\frac{1}{\kappa}\left(\frac{1}{\kappa}m'_1\right)'\right)'-\left(c_0^2+1\right)m'_1=0.
\end{equation}
By means of changing of the independant variable $s$ with $z=\int\kappa ds$, we obtain
\begin{eqnarray}\label{A7}
  \dddot{m}_1-(c_0^2+1)\Dot{m}_1=0.
\end{eqnarray}
If we solve this equation we get 
\begin{eqnarray}\label{A8}
  m_1=\frac{1}{\sqrt{c_0^2+1}}\left(a_1e^{\sqrt{c_0^2+1}z}-a_2e^{\sqrt{c_0^2+1}z}\right)
\end{eqnarray}
From (\ref{A1}) we obtain $m_2=\Dot{m}_1$ and $m_3=\frac{1}{c_0}(\Ddot{m}_1+m_1)$.
\end{proof}
\textbf{Subcase 2: $m_1=0$}\\
With the same computation as above, we have the following theorem:
\begin{theorem}
Let $\left(\alpha; \beta\right)$ be a curve pair of constant breadth in $(M, g_f)$. If $\alpha$ is
a spacelike asymptotic curve (with $\epsilon_2=-1$ and $\epsilon_3=1$) lying in a timelike surface in $M_f$. If $m_1=0$, then the curve $\beta$ is given by
\begin{displaymath}
  \beta(s^*) = \alpha(s) + \left[b_1\cos\left(\int\tau ds\right)+b_2\sin\left(\int\tau ds\right)\right]Y(s)+\left[-b_1\sin\left(\int\tau ds\right)\\
  +b_2\cos\left(\int\tau ds\right)\right]U(s).
\end{displaymath}
\end{theorem}
\subsubsection{Case (For Principal line)}
In this case we have the two following theorems:
\begin{theorem}
Let $(\alpha, \beta)$ be a pair curves of constant breadth in $(M, g_f\epsilon)$. Let $\alpha$ be spacelike principal line (with $\epsilon_2=-1$ and $\epsilon_3=1$) and a general helix then $\beta$ is given by 
\begin{eqnarray}
  \beta(s^*)=\alpha(s)+m_1T+m_2Y+m_3U,
\end{eqnarray}
where 
\begin{eqnarray*}
  m_1=\frac{1}{\sqrt{1+c^2}}\left(a_1e^{\sqrt{1+c^2}\theta}-a_2e^{-\sqrt{1+c^2}\theta}\right),
\end{eqnarray*}
$m_2=c\int m_1\cosh\theta d\theta$ and $m_3=-c\int m_1\sinh\theta d\theta$.
\end{theorem}
\begin{theorem}
Let $(\alpha, \beta)$ be a pair curves of constant breadth in $(M, g_f\epsilon)$. Let $\alpha$ be principal line (with $\epsilon_2=-1$ and $\epsilon_3=1$) lying in a timelike surface in $M_f$. If $m_1=0$ then $\alpha$ is general helix or $\alpha$ is planar curve and the curve $\beta$ is expressed as
\begin{eqnarray}
  \beta(s^*)=\alpha(s)+c_2Y+c_3U,
\end{eqnarray}
where $c_2$ and $c_3$ are constants.
\end{theorem}

\subsection*{Acknowledgments}
The author would like to thank the anonymous Referees for their comments and suggestions.
All many thanks to professor Ferdag Kahraman from Ahi Evran University (Turkish)  for their remarks and suggestions.

\bibliographystyle{amsplain}

\end{document}